\documentclass[11pt,english]{article}

\usepackage[margin = 2.25 cm,bottom=18mm,footskip=7mm,top=19mm]{geometry}

\usepackage{amsthm}
\usepackage{amsmath}
\usepackage{amssymb}
\usepackage{setspace}
\usepackage{mathtools}
\usepackage{graphicx}
\usepackage[hidelinks]{hyperref}
\usepackage{thm-restate}
\usepackage{cleveref}
\usepackage{enumitem}
\usepackage{framed}
\usepackage{subcaption}
\PassOptionsToPackage{hyphens}{url}\usepackage{hyperref}
\usepackage[square,sort,comma,numbers]{natbib}
\usepackage{floatrow}

\theoremstyle{plain}

\newtheorem*{thm*}{Theorem}
\newtheorem{thm}{Theorem}
\Crefname{thm}{Theorem}{Theorems}

\newtheorem*{lem*}{Lemma}
\newtheorem{lem}[thm]{Lemma}
\Crefname{lem}{Lemma}{Lemmas}

\newtheorem*{claim*}{Claim}

\crefname{claim}{Claim}{Claims}
\Crefname{claim}{Claim}{Claims}

\Crefname{prop}{Proposition}{Propositions}

\newtheorem*{prop*}{Proposition}

\crefname{cor}{Corollary}{Corollaries}

\newtheorem{conj}[thm]{Conjecture}
\crefname{conj}{Conjecture}{Conjectures}

\newtheorem{qn}[thm]{Question}
\Crefname{qn}{Question}{Questions}

\Crefname{obs}{Observation}{Observations}

\newtheorem{ex}[thm]{Example}
\Crefname{ex}{Example}{Examples}

\theoremstyle{definition}

\Crefname{prob}{Problem}{Problems}

\Crefname{defn}{Definition}{Definitions}

\newtheorem*{defn*}{Definition}

\theoremstyle{remark}

\renewenvironment{proof}[1][]{\begin{trivlist}
\item[\hspace{\labelsep}{\bf\noindent Proof#1.\/}] }{\qed\end{trivlist}}

\newcommand{\remove}[1]{}

\newcommand{\floor}[1]{
    \left \lfloor #1 \right \rfloor
}

\newcommand{\ceil}[1]{
    \left \lceil #1 \right \rceil
}

\newcommand{\eps}{\varepsilon}
\renewcommand{\P}{\mathbb{P}}
\newcommand{\R}{R_{\ell}}
\renewcommand{\ex}{\mathop{\mathrm{ex}}}
\newcommand{\mc}{monochromatic}

\begin{document}

%\allowdisplaybreaks

\title{\vspace{-0.9cm} List Ramsey numbers}

\author{
Noga Alon\thanks{Department of Mathematics, Princeton
University, Princeton,
New Jersey, USA and Schools of Mathematics and
Computer Science,
Tel Aviv University, Tel Aviv, Israel. 
e-mail: \texttt{nogaa}@\texttt{tau.ac.il}.
Research supported in part by
NSF grant DMS-1855464, ISF grant 281/17,
GIF grant G-1347-304.6/2016,
BSF grant 2018267
and the Simons Foundation.
}
	\and
	Matija Buci\'c\thanks{
	    Department of Mathematics, 
	    ETH, 
	    8092 Zurich;
	    e-mail: \texttt{matija.bucic}@\texttt{math.ethz.ch}.
	    Research supported in part by SNSF grant 200021-175573.
	}
	\and
	Tom Kalvari \thanks{
School of Mathematics,
Tel Aviv University, Tel Aviv, Israel. 
e-mail: \texttt{tomkalva}@\texttt{mail.tau.ac.il}.
}
	\and
    Eden Kuperwasser \thanks{
School of Mathematics,
Tel Aviv University, Tel Aviv, Israel. 
e-mail: \texttt{kuperwasser}@\texttt{mail.tau.ac.il}.
    }
    \and
	Tibor Szab\'o\thanks{Institute of Mathematics, FU Berlin,
          14195 Berlin; e-mail: \texttt{szabo}@\texttt{math.fu-berlin.de}.
Research supported in part by GIF grant No. G-1347-304.6/2016 and by
the Deutsche Forschungsgemeinschaft (DFG, German Research Foundation)
under Germany's Excellence Strategy -- The Berlin Mathematics Research Center MATH+ (EXC-2046/1, project ID: 390685689).
}
}

\date{}

\maketitle

\begin{abstract}
    \setlength{\parskip}{\smallskipamount}
    \setlength{\parindent}{0pt}
    \noindent
We introduce a list colouring extension of classical Ramsey numbers. 
We investigate when the two Ramsey numbers are equal, and in general,
how far apart they can be from each other. We find graph sequences
where the two are equal and where they are far apart. 
For $\ell$-uniform cliques we prove that the list Ramsey number is 
bounded by an exponential function, while it is well-known that the
Ramsey number is super-exponential for uniformity at least $3$.
This is in great contrast to the graph case where we cannot even decide 
the question of equality for cliques.

\end{abstract}

\section{Introduction}
The notion of proper colouring and the corresponding parameter of the
chromatic number is one of the most applicable and
widely-studied topics in (hyper)graph theory. In some of these applications
the list-colouring extension of the notion is necessary to describe 
the situation appropriately. 
A {\em colouring} of a hypergraph $H= (V,E)$ is a function
$c:V\rightarrow \mathbb{N}$. 
A colouring is called  {\em proper} if no hyperedge $e\in E$ is
monochromatic. For an assignment $L:V\rightarrow  2^{\mathbb{N}}$ of 
a subset $L_v\subseteq \mathbb{N}$ of colours to each vertex
$v\in V$, we call a colouring $c:V\rightarrow \mathbb{N}$ an 
{\em $L$-colouring} if $c(v)\in
L_v$ for every $v\in V$. When $L_v = [k]$ for every $v\in V$, an
$L$-colouring is called a {\em $k$-colouring}.

The chromatic number $\chi(H)$ is the smallest integer $k$ such that there
exists a proper $k$-colouring of $H$ and the list-chromatic number (or
choice number)  $\chi_{\ell} (H)$ is the smallest integer $k$ such that for
every assignment $L$ of lists of size $k$ to the vertices of $H$ 
there is a proper $L$-colouring.
By definition $\chi (H) \leq \chi_{\ell} (H)$ for every graph
$H$. Under what circumstances are the two
parameters equal and how far they can be from each other?
 These fundamental
questions are the subject of vigorous research, see, e.g., 
\cite{BM}, Chapter 14 and the references therein. A notorious
open question in this direction is the List Colouring Conjecture 
suggested independently by various researchers including Vizing,
Albertson, Collins,
Tucker and Gupta, which appeared first in print in the
paper of Bollob\'as and Harris \cite{BH}
and states that the list-chromatic number is equal to the
chromatic number for line-graphs. This conjecture was proved by Galvin
\cite{Galvin} for bipartite graphs, by H\"aggkvist and Janssen
\cite{haggkvist} for cliques of odd
order, by Alon and Tarsi \cite{Al} 
for cubic bridgeless planar graphs, by Ellingham and Goddyn 
\cite{EG} for regular class-$1$ planar multigraphs
and by Kahn \cite{Ka} 
asymptotically, but is very much open in general. Even
for cliques $K_n$ of even order it is not known whether the
list-chromatic number of its line graph is
$n$ or $n-1$.
%$\chi(L(H))=\chi_{\ell} (L(H))$  

A particularly interesting instance of hypergraph colouring arises from Ramsey theory, which is concerned with the proper colouring of very specific hypergraphs. 
Ramsey's Theorem states that for any $r$-uniform hypergraph
(or $r$-graph) $G$ and number $k$ of colours any $k$-colouring of the
$r$-subsets of $[n]$ contains a monochromatic copy of the hypergraph
$G$, provided $n$ is large enough depending on $G$ and
$k$. The smallest such integer $n$ is usually called the $k$-colour
Ramsey number of the hypergraph $G$. 

	\begin{defn*}
The $k$-colour (ordinary) \textit{Ramsey number} of an $r$-graph $G$ is defined as 
	$$R(G,k) : =\min\{n \mid \forall \text{$k$-colouring of } E(K_n^{(r)}),
        \text{ $\exists$ a \mc{} copy of } G\}.$$
	\end{defn*}

The study of Ramsey numbers has attracted a lot of attention over the 
years and many natural generalisations and extensions of 
Ramsey numbers were considered, for excellent surveys see
\cite{erdos1}, \cite{ramsey-survey} and the 
references therein. In this paper we study a new variant, a list colouring version of the Ramsey problem. 
In particular, when is it possible to assign lists of size $k$ 
to the edges 
of $K_n^{(r)}$ in such a way that if we colour each edge with a colour
from its list we can always find a monochromatic copy of a given
graph. If we require all lists to be the same we recover the ordinary
Ramsey number. 
This gives rise to the following list-colouring variant of the Ramsey number. 
		
	\begin{defn*}
The $k$-colour \textit{list Ramsey number} of an $r$-uniform hypergraph $G$ is defined by 
	\begin{align*}
	\R(G,k)=\min \{n \mid & \exists L: E(K_n^{(r)}) \rightarrow
                                \binom{\mathbb{N}}{k}
                                \text{ s.t. $\forall$ 
                                $L$-colouring of $E(K_n^{(r)})$
                                $\exists$ a \mc{} copy of } G\}.
	\end{align*}
	\end{defn*}

A first observation, immediate from the definition, is that for
every $G$ and $k$, we have  
	\begin{align}\label{ineq:trivial}
	    \R(G,k) \le R(G,k).
	\end{align}

In our paper we will be investigating when this inequality is an equality
and, more generally, when the two quantities are close to each other
and when they are far apart, how far apart can they be. This question for specific families of graphs turns out to be related to several long standing open problems such as the aforementioned list colouring conjecture, we give the details in the following subsections.

{\bf Remark.} Notion of the list Ramsey number was suggested 
at \href{https://mathoverflow.net/questions/298778/list-ramsey-numbers}{https://mathoverflow.net/questions/}
\href{https://mathoverflow.net/questions/298778/list-ramsey-numbers}{298778/list-ramsey-numbers}, 
where some basic observations were made, as well as a conjecture, 
which we disprove,  that 
inequality \eqref{ineq:trivial} is actually always an equality.

\subsection{Results}
\subsubsection{Stars}
Any edge-colouring of a graph contains no monochromatic copy of $K_{1,2}$ if and
only if it is proper. Therefore the $k$-colour Ramsey number (list
Ramsey number) of $K_{1,2}$
is equal to the smallest 
number $n$ such that $\chi'(K_{n}) > k$ 
($\chi'_{\ell}(K_n) > k$, respectively), where 
here $\chi'(G)$ denotes the edge chromatic number of $G$ which can be defined as the chromatic number of its line graph and similarly for $\chi'_{\ell}$.
Hence the question whether the two Ramsey
numbers of $K_{1,2}$ are equal for an arbitrary number $k$ of colours 
is essentially equivalent to the aforementioned List Colouring Conjecture for cliques.  
It was proved by H\"aggkvist and Janssen that $\chi_{\ell}'(K_n) \leq
n$ for every $n$, which implies that the list chromatic index $\chi'_{\ell}(K_n)$ is
equal to the chromatic index $\chi'(K_n)$ for odd $n$. The question whether  
$\chi_{\ell}' ( K_n)$ is equal to $\chi'(K_n)$ for even $n$ is still
open.
Consequently we know that $\R(K_{1,2}, k) = k +1 = R(K_{1,2}, k)$
when
$k$ is even, but we do not know whether $\R(K_{1,2}, k)$ is $k+1$ or $k+2$
when $k$ is odd.

The multicolour Ramsey number for stars of arbitrary size 
was determined by Burr and Roberts~\cite{stars}. They
showed that
\begin{equation} %\label{eq:burr-roberts}
(r-1)k +1\leq R(K_{1,r},k) \leq (r-1)k+2,
\end{equation}
 and that the lower bound is tight if and only if both $r$ and $k$ are even.

In our first theorem we extend the validity of the lower
bound to the list Ramsey number, thus establishing that the lower
bound is tight when both $r$ and $k$ are even. Furthermore, we show for any fixed number $k$ of colours, that for large enough $r$ the upper bound is tight.
\begin{thm} \label{thm:stars-lower-bound-k-col}
For any $k$ and $r\in \mathbb{N}$, except
    possibly finitely many integers $r$ for each odd $k$, we have
    $\R(K_{1,r}, k)= R(K_{1,r}, k)$. More precisely,
\begin{itemize} 
\item[(a)] For every $r,k\in \mathbb{N}$, we have \begin{equation}
    \label{eq:star-lower} (r-1)k + 1\leq  \R(K_{1,r}, k).
\end{equation} 
In particular, $\R(K_{1,r}, k) = (r-1)k +1= R(K_{1,r}, k)$
  whenever both $r$ and $k$ are even.
\item[(b)]For every $k\in \mathbb{N}$ there exists $w(k)\in \mathbb{N}$
  such that the following holds. For every $k$ and $r \geq w(k)$
  that are not both even, we have $$\R(K_{1,r}, k) = (r-1)k +2 = R(K_{1,r},k).$$
\end{itemize}
\end{thm}
Our theorem fails to give a full characterisation of the tightness of
the lower bound in \eqref{eq:star-lower}. For two colours we can give such a
characterisation and find that the two Ramsey numbers
are always equal.

\begin{restatable}{thm}{thmstars}\label{thm:stars-2-col}
For every $r\in \mathbb{N}$ we have
$$\R(K_{1,r},2)=R(K_{1,r},2)=
\begin{cases}
2r-1 \text{ if } r \text{ is even}\\
2r\textcolor{white}{-11 } \text{ if } r \text{ is odd.}
\end{cases}
$$
\end{restatable}

\subsubsection{Matchings}

We saw above that for stars the two Ramsey numbers are equal, possibly
up to an additive constant one. Next we consider matchings and find
that, unlike for stars,  
the ordinary Ramsey number is significantly larger than the list Ramsey number
for most values of the parameters.

Ramsey numbers of matchings were determined in 1975 by Cockayne and
Lorimer \cite{cockayne_lorimer_1975}. They showed that for every $r,
k\in \mathbb{N}$, 
\begin{equation}\label{eq:cockayne-lorimer}
R(rK_2,k)=rk+r-k +1.
\end{equation}

A trivial lower bound on the list Ramsey number $\R(rK_2, k)$ is $2r$:
if we were to find a matching of size $r$ in $K_n$, monochromatic or
not, then $n$ better be at least the number of vertices in
$rK_2$. It turns out that if the number $k$ of colours is not too large
compared to $r$, then this trivial lower bound 
is asymptotically tight! That is, even
if $n$ is just slightly larger than $2r$, there exists an assignment of lists
of size $k$ to the edges of $K_n$, such 
that any list-colouring of the edges contains
a monochromatic $rK_2$ (i.e., an almost perfect 
matching which is monochromatic). 
Note that by \eqref{eq:cockayne-lorimer}, using the same $k$ colours on
each edge one {\em can} colour 
a much larger clique without a monochromatic $rK_2$. In particular
we show that for any fixed number $k$ of colours 
the two Ramsey numbers are a constant factor $\frac{k+1}{2}$ 
away from each other asymptotically, as $r$ tends to infinity.

The number $k$ of colours becomes more visible in the value of the list
Ramsey number once $k$ is larger than a logarithmic function of the
size $r$ of the matching. In particular for any fixed $r$, we determine the growth
rate of the $k$-colour list Ramsey number up to an absolute
constant factor and find that the ratio of the two Ramsey numbers grows
as $\Theta(\log k)$.

\begin{thm} \label{thm:matchings}
For any fixed $k\ge 2$ and $r$ tending to infinity, we have
$\R(rK_2,k)=2r+o(r)$. In particular 
$$\frac{R(rK_2, k)}{\R(rK_2, k)} = \frac{k+1}{2} + o(1).$$
For any fixed $r\geq 1$ and $k$ tending to infinity, we have 
$\R(rK_2,k)=\Theta(k/\log k).$
In particular  $$\frac{R(rK_2, k)}{\R(rK_2, k)} = \Theta( \log k ).$$
\end{thm}	
In fact we determine the list Ramsey number of matchings for all values of $r$ and $k$ up to a constant factor and when $r$ is sufficiently bigger than $k$ even up to an additive lower order term. For more details 
see Subsection \ref{subs:matchings}. 
\subsubsection{Cliques and Hypergraphs} 
Some of the most famous open problems in Ramsey theory involve cliques.
The proofs of the classic probabilistic lower bounds on $R(K_r, 2)$
all go through in the list chromatic setting, hence 
$$2^{r/2} < \R(K_r, 2) \leq R(K_r, 2) < 2^{2r}.$$ 
Not unexpectedly, we cannot improve on the lower bound. It is
not difficult to see that $\R(K_3, 2) = 6= R(K_3, 2)$,  but for 
$r > 3$ we cannot even decide the equality of the two Ramsey numbers of
$K_r$ when $k=2$. 

For hypergraphs of uniformity $\ell \geq 3$ however, we are able to 
show an exponential (or even larger, depending on the uniformity) 
separation  between the 
ordinary and the list Ramsey numbers.
On the one hand it is known via the stepping-up lemma of
Erd\H{o}s and Hajnal (see Chapter 4.7 of \cite{erdos1}) that the
Ramsey numbers of cliques are super-polynomial in the exponent whenever $\ell\ge
4$ or $\ell=3,k\ge 3$ (Conlon, Fox, Sudakov \cite{conlon-ramsey} for
$k = \ell= 3$) and in fact grow at least as fast as a tower of
height
$\ell-2$. For the list Ramsey number on the other hand we can show that for 
fixed uniformity and number of colours it is upper bounded by an
exponential in a polynomial in $r$. 
%T: I would state this theorem only for growing cliques (the statement
%for other $H$ is immediate) and immediately here, after the
%discussion where we promise the separation. 
\begin{restatable}{thm}{thmubb}
\label{thm:cliques-r-grows} For arbitrary positive integers $r \geq
\ell$ and $k\in \mathbb{N}$ we have
$$\R(K_r^{(\ell)},k) \le 2^{4r^{3\ell-1}+4kr^{\ell-1}\log_2 r}.$$
\end{restatable}

This theorem obviously provides an upper bound on the list Ramsey number of
any fixed $\ell$-graph $H$, which is an exponential function of $k$.
For a growing number of colours the base of the exponent can be
strengthened. 
In order to state our result, we need to introduce a few standard
parameters. Let $\ex(H,n)$ denote the maximum number of edges in 
an $H$-free $\ell$-graph on $n$ vertices and let 
$\pi(H)=\lim_{n\to \infty}\ex(H,n)/\binom{n}{\ell}.$ 
Assuming $H$ has at least $2$ edges let 
$$m(H)=\underset{H' \subset H, e(H') > 1}{\max}\frac{e(H')-1}{v(H')-\ell}.$$

\begin{restatable}{thm}{thmub} 
\label{thm:hypergraph-upper}%
Let $H$ be an $\ell$-uniform hypergraph. Then, as $k$ tends to infinity, we have
$$\R(H,k) \le (1-\pi(H)+o(1))^{-km(H)}.$$
\end{restatable}

For the particular case of $k$-colour list Ramsey number of
the triangle the theorem gives the exponential upper bound 
$\R (K_3, k) \leq (4+o(1))^k$. 

The behaviour of the 
ordinary $k$-colour Ramsey number $R(K_3,k)$ is related to 
other open problems, most notably the question if the maximum
possible Shannon capacity of a graph with independence number $2$
is finite, see \cite{EMT}, \cite{AO}. It is one of the notorious open
problems of combinatorics to decide whether its growth rate is
exponential or superexponential. Erd\H os offers \$100 for its
resolution and \$250 for the determination of the limit
$\lim\limits_{k \to \infty} \sqrt[k]{R(K_3,k)}$ provided it exists. 
The current best lower bound is $R(K_3,k) \geq 3.199^k$ 
(see \cite{XZER}), so not
large enough for us to conclude that the ordinary and the list Ramsey
numbers are different.

For the list Ramsey number we can only give a much weaker lower bound,
where the exponent is the square root of the number of colours.

\begin{restatable}{thm}{thmlb} \label{thm:hypergraph-lower}
If $H$ is an $\ell$-uniform hypergraph with $\chi(H) >r$, then we have 
$$\R(H,k) \ge e^{\sqrt{k\log r/(4\ell)}}.$$
In particular $\R(K_3,k) > e^{\sqrt{k}/4}$.%e^{\frac{\sqrt{\log 2}}{2\sqrt{2}}\sqrt{k}}$.
\end{restatable}

Note that this theorem gives a lower bound exponential in the square
root of $k$ for every non-$2$-colourable $\ell$-graph $H$. 
Our argument extends to every non-$\ell$-partite $\ell$-graph,
even if they are $2$-colourable, with a somewhat worse constant factor
in the exponent.  

\begin{thm} \label{thm:non-l-partite}
Let $H$ be an $\ell$-uniform hypergraph which is not $\ell$-partite. We have 
$$\R(H,k) \ge e^{c_\ell\sqrt{k}},$$
where $1/c_\ell=2\ell e^{\ell/2}.$
\end{thm}

Our proof method works most efficiently when the ordinary
Ramsey number of $H$ is small.  It is known that 
the multicolour Ramsey number of an $\ell$-graph $H$ is
polynomial in $k$ if and only if $H$ is $\ell$-partite. For them we
determine the list Ramsey number up to a poly-logarithmic factor.

\begin{thm}\label{thm:l-partite}
Let $H$ be an $\ell$-partite $\ell$-uniform hypergraph with parts of size at most $r.$ There is a constant $c=c(r,\ell)$ such that 
$$R(H,\floor{ck/\log k})\le \R(H,k)\le R(H,k).$$
In particular, if $\ex(H,n)=\tilde{\Theta}(n^{\ell - \eps(H)}),$\footnote{Here $f=\tilde{\Theta}(g)$ means, as usual, that
$f$ and $g$ are equal up to polylogarithmic factors.}
for some $\eps(H) >0,$ then
$$\R(H,k) =\tilde{\Theta}(R(H,k))=\tilde{\Theta}(k^{1/\eps(H)}).$$
%In particular, if $\ex(H,n) \leq n^{\ell -\eps(H)}$, then $$\frac{R(H,k)}{\R (H,k)} = O\left( \log^{1/\eps(H)} k\right)$$. 
\end{thm}

This theorem can be considered an extension of the second part of 
Theorem~\ref{thm:matchings}, where we determine that the ordinary and
the list Ramsey numbers of matchings are exactly a
$\log k$ factor away from each other. 
For several bipartite graphs (for example for complete 
bipartite graphs $K_{r, s}$ for $s > (r-1)!,$ even cycles $C_6$ and $C_{10}$ or general trees) the asymptotic behaviour of the ordinary Ramsey number 
is known up to a polylogarithmic factor and hence by \Cref{thm:l-partite} so is
the list Ramsey number.

The rest of this paper is organised as follows. In Subsection 2.1 
we prove our results for stars. In Subsection 2.2 we prove the 
results for matchings, demonstrating on a relatively simple example 
the methods we are going to use in Subsection 2.3 to prove the 
bounds for list Ramsey numbers of general graphs. In Section 3 we give concluding remarks and present some open problems. All our logarithms are natural unless explicitly indicated otherwise.

\section{Bounds for list Ramsey numbers}

\subsection{Stars}
Let us start with a few preliminaries and tools which we will use throughout this subsection.
\begin{thm}[Galvin \cite{Galvin}]\label{thm:galvin}
If $G$ is a bipartite graph of maximal degree $\Delta$ 
then $\chi'_{\ell}(G)=\Delta.$
\end{thm}

To show $\R(G,k) > n$ we need to show that for any assignment of lists of size $k$ to the edges of $K_n$ we can choose the colours from the lists in such a way that we create no monochromatic copy of $G.$ We distinguish two cases depending on parity. The following simple observation will enable us to give lower bounds on $\R(K_{1,r},k).$

\begin{lem}\label{thm:star-lower-bound}
Let us assume that graphs $G_1,\ldots, G_t$ partition the edge set of 
$K_n.$ If $\chi'_{\ell}(G_i)\le k$ for all $i$ and each vertex 
belongs to at most $r-1$ of the $G_i$'s then $\R(K_{1,r},k)>n.$ 
\end{lem}
\begin{proof}
    Let $L$ be an assignment of lists of size $k$ to the edges of $K_n.$ 
By the assumption that $\chi'_{\ell}(G_i) \le k$ there is a 
proper $L$-colouring $c_i$ of each $G_i.$ Let us define an $L$-colouring $c$ of $E(K_n)$ by $c(e)=c_i(e),$ where $i$ is the index of the unique $G_i$ containing $e.$ Note that since any vertex $v$ belongs to at most $r-1$ $G_i$'s we know that edges incident to $v$ are using colours from at most $r-1$ $c_i$'s. Since each $c_i$ is proper this means that for any fixed colour $v$ is incident to at most $r-1$ edges of this colour, showing there can be no monochromatic $K_{1,r}$ under $c$ as desired.
\end{proof}

We begin with the case of $2$ colours, by proving \Cref{thm:stars-2-col}.
\thmstars*
\begin{proof}
It is well known that the standard Ramsey number satisfies the same equalities, \cite{stars}. So by \eqref{ineq:trivial} we only have to show the corresponding lower bounds.

\subsection*{Case 1: $r$ even.}
We will make use of the following fact proved by Alspach and Gavlas \cite{alspach}. 
\begin{prop*}
Let $n$ be an even integer. $K_n$ can be partitioned into a single perfect matching and Hamilton cycles. 
\end{prop*}
Let $n=2r-2.$ By the above proposition we can partition $K_n$ into a perfect matching $G_1$ and $r-2$ Hamilton cycles $G_2,\ldots, G_{r-1}$. By Galvin's theorem \cite{Galvin} we know that $\chi'_{\ell}(G_i) \le 2$ and each vertex belongs to exactly $r-1$ of the $G_i$'s so by \Cref{thm:star-lower-bound} we are done.

\subsection*{Case 2: $r$ odd.}
In this case we make use of a different partitioning result of Alspach and Gavlas \cite{alspach}. 
\begin{prop*}
Let $n$ be an odd integer and $m$ an integer satisfying $4 \le m \le n.$ $K_n$ can be partitioned into cycles of length $m$ if and only if $m$ divides the number of edges of $K_n$. 
\end{prop*}

Let $n=2r-1.$ Let us first assume that $r \ge 5.$ Since $|E(K_n)|=n(n-1)/2=(2r-1)(r-1)$ by the above result we can partition $K_n$ into cycles of length $r-1.$ Since $r$ is odd these cycles are 
bipartite and have $\chi'_{\ell}(C_{r-1})=2.$ As they are $2$-regular and 
partition $E(K_n)$ we know that each vertex belongs to exactly $r-1$ of these cycles. Therefore, we are done by \Cref{thm:star-lower-bound}.

The case $r=1$ is immediate, so we are left with the case $r=3.$ Let $L$ be an assignment of lists of size $2$ to the edges of $K_5.$ 
Partition $K_5$ into two $5$-cycles $C_1,C_2.$ If we can properly colour both $C_1$ and $C_2$ using colours from the lists we are done. It is well-known and easy to see that the only way in which a $5$-cycle does not admit a $2$-colouring from its lists is if the lists are all the same. Therefore, we may assume that edges of one cycle, say $C_1$ have the same lists. We now colour all edges of $C_1$ using a single colour $c$ from their list 
and colour all edges of $C_2$ by arbitrary colours in their lists
which differ from $c.$ 
In this colouring there is no monochromatic $K_{1,3}$ as desired.
\end{proof}

Let us now consider the case of more colours. As in the case of $2$-colours all our upper bounds come from the ordinary Ramsey numbers, which were determined by Burr and Roberts in \cite{stars} and the trivial inequality \eqref{ineq:trivial}. The following two lemmas establish the two lower bounds claimed in \Cref{thm:stars-lower-bound-k-col}, completing its proof.

\begin{lem}
$$(r-1)k+1 \le \R(K_{1,r},k).$$
\end{lem}
\begin{proof}
Let $n=(r-1)k,$ partition the vertices of $K_n$ into sets $V_1,\ldots,V_{r-1},$ each of size $k.$ 
We let $G_i$ be the subgraph induced by $V_i$ and for $i\neq j$ 
we let $G_{i,j}$ be the complete bipartite subgraph 
with parts $V_i,V_j.$ By \Cref{thm:galvin} we know that 
$\chi'_{\ell}(G_{i,j})\le k$ and since by a result of H\"aggkvist and 
Janssen \cite{haggkvist} we know that $\chi'_{\ell}(K_k) \le k$ we also have 
$\chi'_{\ell}(G_i)\le k.$ 
Every vertex belongs to exactly $r-1$ of these subgraphs which 
partition $E(K_n)$, and we are done by \Cref{thm:star-lower-bound}.
\end{proof}

This completes the proof of \Cref{thm:stars-lower-bound-k-col} part (a). Before turning to part (b) we state a packing result of Gustavsson \cite{gustavsson} which we will use for its proof.
\begin{thm*}[Gustavsson~\cite{gustavsson}] 
For any graph $F$ there exists an $\eps = \eps(F) > 0$ and
$n_0 = n_0(F)$ such that for any graph $G$ on $n \ge n_0$ vertices with minimum
degree at least $(1 -\eps)n$ one can partition the edge set of $G$ into copies of $F$, provided: 
\begin{itemize}
\item $e(F)\mid e(G)$ and
\item $\gcd(F) \mid  \gcd(G),$ 
\end{itemize}
where $e(H)$ denotes the number of edges of a graph $H$ and 
$\gcd(H)$ denotes the greatest common divisor of the degrees 
of vertices in $H$. 
\end{thm*}

\begin{lem}
For every $k\in \mathbb{N}$ there exists $w(k)\in \mathbb{N}$
such that the following holds. For every $k$ and $r \geq w(k)$
that are not both even, we have $$\R(K_{1,r}, k) = (r-1)k +2.$$
\end{lem}

\begin{proof}
Let $n=(r-1)k+1.$ Therefore, $e(K_n)=\binom{n}{2}=(r-1)k((r-1)k+1)/2.$ Since, if $k$ is even $r$ must be odd and in particular, $2|r-1$ we know that $k \mid e(K_n).$ Let $t \equiv e(K_n)/k \bmod k.$ Let $G_1, \ldots, G_t$ be vertex disjoint subgraphs of $K_n$ each isomorphic to $K_{k+1,k+1}$ with a perfect matching removed, where we require $w(k) \ge 2k+1$ in order to have enough room ($n=(r-1)k+1 \ge (w(k)-1)k+1\ge 2k^2+1 \ge 2(k+1)t$, since $t \le k-1$). Let $G$ be the subgraph of $K_n$ obtained by removing 
the edges of all $G_i$'s and let $F=K_{k,k}.$ Note that $e(G) \equiv tk-tk(k+1) \equiv 0 \bmod k^2$ so $e(F)=k^2 \mid e(G).$ Furthermore, every vertex of $K_n$ not in any $G_i$ still has degree $(r-1)k$ in $G$ while any vertex of $G_i$ has degree $(r-1)k-k$ so $\gcd(G)=k=\gcd(F).$ Therefore, if we let $\eps=\eps(K_{k,k})$ and $n_0=n_0(K_{k,k})$ given by the above 
theorem, then for $w(k) \ge \max(n_0/k, 2/\eps)$ the above theorem applies, implying that $E(G)$ can be partitioned into $G_{t+1},\ldots, G_{q}$ all isomorphic to $K_{k,k}$. 

Since each $G_i$ is a $k$-regular bipartite graph Galvin's Theorem implies $\chi'_{\ell}(G_i) \le k$ and since $G_i$'s partition $E(G)$ we know that each vertex belongs to at most $(n-1)/k=r-1$ of the $G_i$'s so our \Cref{thm:star-lower-bound} applies and implies the result. 
\end{proof}

\subsection{Matchings}\label{subs:matchings}
%As mentioned in the introduction the ordinary Ramsey numbers ofmatchings were determined in 1975 by Cockayne and Lorimer \cite{cockayne_lorimer_1975} who have shown $R(rK_2,k)=rk+r-k+1$. In this section we determine the list Ramsey numbers of matchings up to a constant factor and conclude that, unlike in the previous subsection, the two Ramsey numbers differ significantly.

%If we were to find a matching of size $r$ in $K_n$, monochromatic or not, then $n$ better be at least $2r$, the number of vertices in $rK_2$. It turns out that if the number $k$ of colours is not too large compared to $r$, then this trivial lower bound of $2r$ on the list Ramsey number is asymptotically tight! That is, even if $n$ is just slightly larger than $2r$, there exists an assignment of lists of size $k$ to the edges of $K_n$, that any list-colouring of the edges creates a monochromatic $rK_2$ (which is an almost perfect monochromatic matching). 

%The number $k$ of colours becomes more visible in the value of the list Ramsey number once $k$ is larger than a logarithmic function of the size $r$ of the matching. In this case we determine the list Ramsey number up to an absolute constant factor.
In this section we will show the following bounds on the list Ramsey number of matchings.
\begin{thm} \label{thm:matching-precise}
Let $r,k \in \mathbb{N}$. 
If $2(k+1) \leq \log r$ then 
\begin{align*}
    2r\le &\: \R(rK_2,k) \le 2r+42r^{k/(k+1)}.
\end{align*} %\quad \text{else }\\
If $2(k+1) > \log r > 0$ then
\begin{align*}    
\frac{rk}{4\log (rk)}\le &\: \R(rK_2,k) \le \frac{34rk}{\log (rk)}.
\end{align*}
\end{thm}
Theorem~\ref{thm:matchings} is now an immediate consequence of
Theorem~\ref{thm:matching-precise} and \eqref{eq:cockayne-lorimer}.

The proof of Theorem~\ref{thm:matching-precise} appears in the
following two lemmas. 
Our arguments below aim to illustrate as well the ideas we apply
for the  general setting in the next subsection, hence they  
are slightly more complicated than necessary. 

We start with the lower bound. 
\begin{lem}\label{lem:lower-bound-mathcings}
Assuming $r,k \in \mathbb{N}$ such that $rk>1$ we have
$$\R(rK_2,k) \ge \max\left(2r,\:\frac{(r-1)k}{2\log (rk)}\right).$$
\end{lem}

\begin{proof}
Let $n=\max\left(2r-1,\:(r-1)\cdot\left\lfloor\frac{k}{2\log
      (rk)}\right\rfloor+r\right).$ 
% Changed slightly the $n$, because I thought this is exactly what we needed.
Our task is to show that for any assignment $L$ of lists of size $k$ to $E(K_n)$ we can choose an $L$-colouring without a monochromatic $rK_2.$
This is clear if the first term of the maximum is greater or equal than the
second, because then $rK_2$ has more vertices than $K_n$. 
So we may assume $\frac{k}{2\log (rk)}\ge 2.$ Let $t=\left\lfloor\frac{k}{2\log (rk)}\right\rfloor\ge 2.$

Let $c:E(K_n) \to [t]$ be a $t$-colouring of $E(K_n)$ without a monochromatic $rK_2,$ which exists 
since $R(rK_2,t)=(r-1)t+r+1>n,$ using \eqref{eq:cockayne-lorimer}.

We split all colours in $\cup_{e\in E(K_n)} L_e$ into $t$ types indexed by $[t]$, with each colour being assigned a type independently and uniformly at random.
Let $B_e$ denote the event that no colour in $L_e$ got assigned the type $c(e)$. Then 
$$\P(B_e)=\left (1-\frac1t\right)^k\le \left (1-\frac{2\log (rk)}k\right)^k \le \frac1{r^2k^2}.$$
So by the union bound we obtain:
$$\bigcup_{e \in E(K_n)} \P(B_e)\le \frac{n^2}{2r^2k^2}<1,$$
where we used $k \ge 2,$ which follows from $\frac{k}{2\log (rk)}\ge 2.$ %(1/2log(rk)+1/k)<1

Thus there is an assignment of types to colours appearing in the lists 
such that for every $e\in E(K_n)$ there is a colour $c'(e)$ of type $c(e)$ in $L_e.$ Note that $c'$ is an $L$-colouring of $K_n$ with no monochromatic $rK_2,$ since otherwise there would be a monochromatic $rK_2$ using only one type of colours, contradicting our choice of $c.$
\end{proof}

We now turn to the upper bounds. 
Once again we need to distinguish between the two regimes.
\begin{lem}\label{lem:upper-bound-mathcings}
Let $r,k \in \mathbb{N}.$ If $2(k+1)\le \log r$ then we have 
$$\R(rK_2,k) \le 2r+42r^{k/(k+1)},$$
and else we have
$$\R(rK_2,k) \le 34rk/\log (rk).$$
\end{lem}
\begin{proof}
First notice that when $r=1$ or $k=1$ the result is immediate, so we
assume $r,k\ge 2$ throughout the proof. In order to show an upper
bound $\R(G,k)\le n$, we need to find a list assignment $L$ of lists
of size $k$ to each edge of $K_n$ in such a way that there is no way
of $L$-colouring $K_n$ without having a \mc 
\ copy of $G.$

Before proceeding with the proof, let us give some intuition for the next step. We are going to choose the lists by assigning each edge a uniformly random subset of colours from a slightly larger universe. Our goal then is to show that with probability less than one our random assignment of lists $L$ has an $L$-colouring having no \mc{} $rK_2$. We now take a union bound over all possible colourings of $K_n$ having no \mc{} $rK_2$ and check what is the probability that a fixed one is an $L$-colouring. The fact every edge misses a random set of colours makes this probability rather low. 

For each edge of $K_n$ we choose independently and uniformly at random a list of size $k$ from the universe $U$ of $k+t$ colours. For now we do not 
specify the values of $n$ and $t$ since they will differ depending on which of the two regimes we are considering, we will however assume that $n$ is even.

Let $B$ denote the event that there is a colouring $c$ from our lists having no \mc{} $rK_2$. Our goal is to show $\P(B)<1.$ Let us restrict attention to the complete bipartite graph $H=K_{n/2,n/2}$ within our $K_n.$ If $B$ happens this means that there is an edge colouring $c$ of $H$ for which every colour class contains no matching of size $r.$ Since $H$ is bipartite K\"onig's theorem implies that every colour class has a cover of size at most $r-1.$

For any subset $S$ of vertices of $H$ 
of size $|S|=r-1$ consider the subgraph of $H$ on the same vertex set containing all the edges of $H$ incident to a vertex in $S$. Denote these subgraphs by $C_1,\cdots C_m,$ where $m=\binom{n}{r-1}.$

The above observation implies that if $B$ happens, every colour class of $c$ on $H$ is completely contained within some $C_i.$ For all $i \in U$ we denote by $c_i$ the subgraph of $H$ made by the $i$-th colour class of $c$. Then
\begin{align}
\P(B) & \le \P(\exists \text{ an $L$-colouring }c:E(H)\to U \text{ s.t. } \forall i\in U, \: \exists j\in[m]: c_i \subseteq C_j) \nonumber \\
 &\le \sum_{j_1, \ldots, j_{k+t} \in [m]} \:\P(\exists \text{ an $L$-colouring }c:E(H)\to U  \text{ s.t. } \forall i\in U: c_i \subseteq C_{j_i}) \nonumber  \\
 & \le m^{k+t} \max_{j_1, \ldots, j_{k+t} \in [m]} \:\P(\exists \text{ an $L$-colouring }c:E(H)\to U \text{ s.t. } \forall i\in U: c_i \subseteq C_{j_i}) \nonumber \\
&= m^{k+t} \max_{j_1, \ldots, j_{k+t} \in [m]} \:\P(\forall e \in E(H), \exists i\in L_e: e \in C_{j_i}). \label{eq:ub}
\end{align}
Let us now bound the last term. For fixed values $j_1, \ldots ,
j_{k+t}$, let $d_e$ denote the number of $C_{j_i}$ to which edge $e$
belongs. As each $C_{j}$ has at most $(r-1)n/2$ edges, we have that
$\sum_{e\in E(H)} d_e \le (k+t)(r-1)n/2.$

\begin{align}
    \P(\forall e \in E(H), \exists i\in L_e: e \in C_{j_i})
    & = \prod_{e \in E(H)}\P(\exists i\in L_e: e \in C_{j_i}) \nonumber \\
    & = \prod_{e \in E(H)} \left(1-\binom{k+t-d_e}{k}/\binom{k+t}{k}\right) \nonumber \\
    & = \prod_{e \in E(H)} \left(1-\left(1-\frac{d_e}{k+t}\right) \cdots\left(1-\frac{d_e}{t+1}\right)\right) \nonumber  \\
    & \le \prod_{e \in E(H)} \left(1-\left(1-\frac{d_e}{t+1}\right)^k\right) \nonumber \\
    & \le \left(1-\left(1-\frac{\widetilde{d_e}}{t+1}\right)^k\right)^{n^2/4}, \label{eq:ub2}
\end{align}
where in the first inequality we used the independence of the assignment of lists between edges, $\widetilde{d_e}:=\frac{\sum_{e\in E(H)} d_e}{|E(H)|} \le \frac{(k+t)(r-1)}{n/2}$ and we used Jensen's inequality. Combining \eqref{eq:ub} and \eqref{eq:ub2} we obtain:
\begin{align}
    \P(B) & \le \binom{n}{r-1}^{k+t} \left(1-\left(1-\frac{\widetilde{d_e}}{t+1}\right)^k\right)^{n^2/4}\nonumber \\
    & \le \left( \frac{en}{r-1}\right)^{(r-1)(k+t)} \left(1-\left(1-\frac{2(r-1)(k+t)}{n(t+1)}\right)^k\right)^{n^2/4}. \label{eq:ub3}
\end{align}
At this point we proceed differently depending on the relation between $k$ and $r$. In the first case we will assume $k$ to be significantly smaller than $r$, specifically we assume $2(k+1)\le \log r.$ We choose $t=(k-1)\cdot\ceil{\frac{n}{20r^{k/(k+1)}}}-1$ and our goal is to show that for $n=2r+2\cdot\ceil{20r^{k/(k+1)}}$ we have $\P(B)<1.$
\vspace{-0.5cm}
\begin{align*}
    \log \P(B)
    & \le \log \left(\left( \frac{en}{r-1}\right)^{(r-1)(k+t)} \left(1-\left(1-\frac{2(r-1)(k+t)}{n(t+1)}\right)^k\right)^{n^2/4}\right)\\
    & < \left(1+\log \frac{n}{r-1}\right) (r-1)(k+t)-\frac{n^2}4 \cdot \left(\frac{n-2r}{n}-\frac{k-1}{t+1}\right)^k\\
    & < 6r(k-1)\left(1+\ceil{\frac{n}{20r^{k/(k+1)}}}\right)-\frac{n^2}4\cdot\left(\frac{20r^{k/(k+1)}}{n}\right)^k \\
    & < \frac{12r(k-1)n}{20r^{k/(k+1)}}-r^2\cdot\left(\frac{20r^{k/(k+1)}}{n}\right)^k \\
    %& = \frac{rn}{20r^{k/(k+1)}}\cdot\left(6(k-1)-r\cdot\left(\frac{20r^{k/(k+1)}}{n}\right)^{k+1}\right)\\
    & = \frac{rn}{20r^{k/(k+1)}}\cdot\left(12(k-1)-\left(\frac{20r}{n}\right)^{k+1}\right)\\
    & < \frac{rn}{20r^{k/(k+1)}}\cdot\left(12(k-1)-2.5^{k+1}\right)< 0,
\end{align*}

where in the second inequality for the second term we used $\log (1-x)
\le -x,$ for $x<1,$ and $1-\frac{2(r-1)(k+t)}{n(t+1)}\ge
1-\frac{2r(k+t)}{n(t+1)}=\frac{n-2r}{n}-\frac{2r(k-1)}{n(t+1)} >
\frac{n-2r}{n}-\frac{k-1}{t+1}$ where the last inequality follows
since $n > 2r.$ In the third inequality for the first term we used
$\left(1+\log \frac{n}{r-1}\right) \le 1+\log 88 \le  6.$ In the
fourth inequality we used $1+ \ceil{x}<2x$ when $x>2$ for the first
term and $n > 2r$ for the second, while in the last inequality we used
$\log r \ge 2(k+2),$ to get:
$\frac{20r}{n}\ge\frac{10}{1+20r^{-1/(k+1)} +(2r)^{-1}}\ge 
\frac{10}{1+21/e^2}>2.5.$

For the second case, when $\log r < 2k+2,$ we let $n=2\ceil{16rk/\log (rk)}$ and $t=k$ and use \eqref{eq:ub3} to get:
\begin{align*}
    \log \P(B) & \le \log \left(\left( \frac{en}{r-1}\right)^{(r-1)(k+t)} \left(1-\left(1-\frac{2(r-1)(k+t)}{nt}\right)^k\right)^{n^2/4}\right )\\
    & \le 2rk \log \left( \frac{en}{r-1}\right)+\frac{n^2}4\log \left(1-\left(1-\frac{\log (rk)}{8k}\right)^k\right)\\
    & \le 2rk  (8+\log k) +\frac{n^2}{4}\log \left(1-e^{-\log (rk)/4}\right)\\ 
    & \le2rk (8+\log k)- (rk)^{-1/4}n^2/4\\
    & \le 16rk\left(k^{1/4}- 16(rk)^{3/4}/\log^2 (rk)\right)\\
    & \le 16rk\left(k^{1/4}- (rk)^{1/4}\right)<0,
\end{align*}
where in the first term of the third inequality we used $\log \left( \frac{en}{r-1}\right)\le 1+\log \left(128k \right)\le 8+\log k,$ while in the second term we used $(1-x) \ge e^{-2x},$ given $x \le 1/2,$ with $x=\frac{\log rk}{8k} \le \frac{2k+2+\log k}{8k}\le 1/2.$ In the fifth inequality we used $8+\log k \le 8k^{1/4}$ and in the sixth $\log x \le 4x^{1/4}.$
\end{proof}
%\remark{If $n=2r+q$ the thing we want is: $$6\log \left( \frac{e(2r+q)} {r-1}\right) \cdot (r-1)(k-1)(4r+2q)^{k-1}<q^{k+1}$$}

\subsection{General bounds.}
In this subsection we give our bounds for general graphs and
hypergraphs.

\subsubsection{Upper bounds.}
We start with upper bounds. The idea closely follows the one presented in the previous section with the main distinction that now it is not so easy to find the appropriate sets $C_j$. Note that the only property we required from $C_i$'s is that the edge set  of every graph not containing a copy of $rK_2$ is contained in some $C_i.$ In the general setting we will find such sets by using the container method introduced by Saxton and Thomason \cite{saxton} and Balogh, Morris and Samotij \cite{balogh}. Specifically, we make use of the following theorem (Theorem 2.3 in \cite{saxton}).

\begin{thm}\label{thm:h-free-graphs}
Let $H$ be an $\ell$-graph with $|E(H)| \ge 2$ and let $\eps>0.$ There is a constant $c>0$ such that for any $n \ge c$ there is a collection of $\ell$-graphs $C_1,\ldots, C_m$ on the vertex set $[n]$, such that 
\begin{enumerate}[label=(\alph*), ref=(\alph*)]
    \item \label{itm:1} every $H$-free $\ell$-graph on the vertex set $[n]$ is contained within some $C_i,$ 
    \item \label{itm:2} $|E(C_i)|\le (\pi(H)+\eps)\binom{n}{\ell}$ and
    \item \label{itm:3} $\log m \le cn^{\ell-1/m(H)}\log n.$
\end{enumerate}
\end{thm}

We now give an upper bound on the list Ramsey number for a fixed graph as the number of colours becomes large.

\thmub*
\begin{proof}
We once again choose the lists for each edge uniformly at random out of the universe of $k+t$ colours. As before, $B$ will denote the event that there is a colouring from our lists having no \mc{} $H$. Once again our goal is to show $\P(B)<1.$

Let $\eps>0,$ Theorem \ref{thm:h-free-graphs} provides us with a constant $c=c(\eps,H)$ and a collection of $\ell$-graphs $C_1,\ldots, C_m$ satisfying the conditions \ref{itm:1},\ref{itm:2} and \ref{itm:3}, where we will choose the value of $n\ge c$ later.
We once again obtain as in \eqref{eq:ub}
$$\P(B) \le  m^{k+t} \max_{j_1, \ldots, j_{k+t} \in [m]} \:\P(\forall e \in E(H), \exists i\in L_e: e \in C_{j_i}).$$
Once again for fixed values of $j_i$ we define $d_e$ to be the number of $C_{j_i}$ that contain the edge $e,$ and denote $\widetilde{d_e}=\sum_{e \in E\left(K_n^{(\ell)}\right)} d_e/\binom{n}{\ell}\le (k+t)(\pi(H)+\eps),$ where the last inequality follows from \ref{itm:2}. Once again as in \eqref{eq:ub2} we obtain: 
\begin{align*}
    \P(\forall e \in E(H), \exists i\in L_e: e \in C_{j_i}) \le \left(1-\left(1-\frac{\widetilde{d_e}}{t+1}\right)^k\right)^{\binom{n}{\ell}}.
\end{align*}
We choose $t=\ceil{k/\eps}$, and require $2\eps<1-\pi(H)$ to get
\begin{align*}
    \log \P(B) & \le (k+t)\log m+\binom{n}{\ell}\log \left(1-\left(1-\frac{\widetilde{d_e}}{t+1}\right)^k\right)\\
    & \le (k+t)cn^{\ell-1/m(H)}\log n-\left(1-\frac{(k+t)(\pi(H)+\eps)}{t}\right)^k\binom{n}{\ell}\\
    & \le ck(1+2/\eps)n^{\ell-1/m(H)}\log n-(1-\pi(H)-2\eps)^k\frac{n^\ell}{\ell^\ell}.
\end{align*}

where we used $\frac{(k+t)(\pi(H)+\eps)}{t}\le (1+\eps)(\pi(H)+\eps) \le \pi(H)+2\eps$ where in the last inequality we used $\pi(H)+\eps < 1-\eps.$ The last expression will be less than $0$ provided
\begin{align}\label{eq:ub4}
\frac{ck(1+2/\eps)\ell^\ell}{(1-\pi(H)-2\eps)^k}<n^{1/m(H)}/\log n.
\end{align}
Given $3\eps < 1-\pi(H),$ for large enough value of $k$ this holds for $n=(1-\pi(H)-3\eps)^{-km(H)}$.
\end{proof}
In the above argument it was important that $H$ was fixed, since the constant $c$ coming from the container theorem depends on $H.$ The dependence of $c$ on $H$ is somewhat complicated, but by analysing the proof of \Cref{thm:h-free-graphs} it should be possible to obtain good bounds for various families of graphs. We illustrate this by obtaining an explicit bound on $\R(K_r^{(\ell)},k)$. We start with a slightly weaker version of \Cref{thm:h-free-graphs}, which is a special case of Theorem 9.2 of \cite{saxton}.\footnote{Where we plugged in the explicit values given in their Corollary 3.6 and Theorem 9.3 to obtain our explicit constant.}

\begin{thm}\label{thm:h-free-graphs-2}
Let $H=K_{r}^{(\ell)}$ with $r>\ell$ and let $\delta>0.$ For any positive integer $n$ there exists a collection of $\ell$-graphs $C_1,\ldots, C_m$ on the vertex set $[n]$, such that 
\begin{enumerate}[label=(\alph*), ref=(\alph*)]
    \item \label{itm:1'} every $H$-free $\ell$-graph on the vertex set $[n]$ is contained within some $C_i,$ 
    \item[(b')] \label{itm:2'} Each $C_i$ contains at most $\delta \binom{n}{r}$ copies of $H$ and
    \item[(c)] \label{itm:3'} $\log m \le \frac{1}{\delta}\log \frac{1}{\delta}2^{10\binom{r}{\ell}^2}n^{\ell-1/m(H)}\log n.$
\end{enumerate}
\end{thm}

%Corollary 3.6 of \cite{saxton} (applied to the hypergaph $G$ of copies of $H$,  which is
%$\binom{r}{\ell}$-uniform, has $\binom{n}{r}$ edges and $\binom{n}{\ell}$ vertices) says that Theorem 13 holds with upper bound of part $(c)$ replaced with:
%$800\binom{r}{\ell}!^3 \binom{r}{\ell} \log \frac1\delta \binom{n}{\ell} \tau \log\frac 1\tau,$ provided $\delta(G,\tau)\le \delta/(12\binom{r}{\ell}!),$ where $\delta(G,\tau)$ is the co-degree polynomial. We choose $\tau=12 n^{-1/m(H)} 2^{4\binom{r}{\ell}^2}/\delta $
%for which their Theorem 9.3 implies $\delta(G,\tau)\le 2^{\binom{r}{\ell}^2}r!2^{-4\binom{r}{\ell}^2}\delta/12 \cdot n^r/\binom{n}{r}\le \delta/(12\binom{r}{\ell}!).$

Apart from an explicit constant in part (c) the main difference compared to \Cref{thm:h-free-graphs} is that in the condition (b') rather than bounding the number of edges in each container we bound the number of copies of the forbidden graph $H$ it contains. It is not hard to obtain condition \ref{itm:2} from (b') by making use of the Erd\H{o}s-Simonovits supersaturation lemma, but requiring an explicit constant makes it slightly messy. We start with the standard bound of De Caen on $\ex(K_{r}^{(\ell)},n).$

\begin{thm}[De Caen \cite{de-caen}]\label{thm:de-caen}
$$\ex(K_{r}^{(\ell)},n)\le\left(1-\frac{n-r+1}{n-\ell+1}/\binom{r-1}{\ell-1} \right)\binom{n}{\ell}.$$ 
\end{thm}

We now state the Erd\H{o}s-Simonovits supersaturation lemma, keeping track of the constants.

\begin{thm}[Erd\H{o}s-Simonovits \cite{kovari-sos}]\label{thm:supersaturation}
Let $H$ be an $\ell$-graph with $r$ vertices, $x,\eps>0$ and $m \in \mathbb{N}.$ Given $\ex(H,m) < x\binom{m}{\ell}$ we have that if an $\ell$-uniform hypergraph on $n$ vertices contains at least $\left(x+\eps \right)\binom{n}{\ell}$ edges then it contains more than $\eps\binom{m}{r}^{-1} \binom{n}{r}$ copies of $H.$ 
\end{thm}
%these numbers are given by the standard averaging proof found e.g. here http://staff.ustc.edu.cn/~jiema/ExtrGT2016/March%208.pdf
Combining the last three theorems gives us the following explicit version of \Cref{thm:h-free-graphs} for the complete $\ell$-graph. 

\begin{thm}\label{thm:h-free-graphs-3}
Let $H=K_{r}^{(\ell)}$ with $r>\ell$ and let $\delta>0.$ For any positive integer $n$ there exists a collection of $\ell$-graphs $C_1,\ldots, C_m$ on the vertex set $[n]$, such that 

\begin{enumerate}[label=(\alph*), ref=(\alph*)]
    \item \label{itm:1''} every $H$-free $\ell$-graph on the vertex set $[n]$ is contained within some $C_i,$ 
    \item \label{itm:2''} $|E(C_i)|\le \left(1-\frac23\binom{r-1}{\ell-1}^{-1}\right)\binom{n}{\ell}$ and
    \item \label{itm:3''} $\log m \le 2^{13\binom{r}{\ell}^2}n^{\ell-1/m(H)}\log n.$
\end{enumerate}
\end{thm}
\begin{proof}
    Let $x=1-\frac56\binom{r-1}{\ell-1}^{-1},\eps=\frac16\binom{r-1}{\ell-1}^{-1}$ and $m=6r.$ By \Cref{thm:de-caen} we know that $\ex(H,m)<x\binom{m}{\ell}$ so \Cref{thm:supersaturation} applies showing that any $\ell$ graph on $n$ vertices with more than $(x+\eps)\binom{n}{\ell}=\left(1-\frac23\binom{r-1}{\ell-1}^{-1}\right)\binom{n}{\ell}$ edges contains at least $\delta \binom{n}{r}$ copies of $H,$ where $1/\delta:=\eps^{-1}\binom{m}{r}\le 6\binom{r-1}{\ell-1}\binom{6r}{r}\le 2^{2\binom{r}{\ell}^2}.$
    Using this value of $\delta$ in \Cref{thm:h-free-graphs-2} we obtain the result.
\end{proof}

We are now ready to obtain the bound on $\R(K_r^{(\ell)},k)$ promised
in the introduction.
\thmubb*
\begin{proof}
%The first inequality is trivial since $H \subseteq K_r^{(\ell)}$ so
%we set $H=K_r^{(\ell)}.$ 
We may assume $r>\ell\ge 2$ and $k\ge 2,$ as otherwise the inequality is clearly true. 

Repeating the argument that lead to \eqref{eq:ub4} with $1-\binom{r-1}{\ell-1}^{-1}$ in place of $\pi(H),$ $\eps=\frac13\binom{r-1}{\ell-1}^{-1}$ and using \Cref{thm:h-free-graphs-3} instead of \Cref{thm:h-free-graphs} we obtain that $\R(K_r^{(\ell)},k)\le n$ given:
\begin{align*}
2^{13\binom{r}{\ell}^2}\cdot k\cdot \left(1+6\binom{r-1}{\ell-1}\right)\cdot \ell^\ell\cdot \left(3\binom{r-1}{\ell-1}\right)^{k}<n^{1/m(H)}/\log n.
\end{align*}
Which, using $m(H)=\frac{\binom{r}{\ell}-1}{r-\ell}\le r^{\ell-1}/(l-1)$ holds for  $n=2^{4r^{3\ell-1}+4kr^{\ell-1}\log_2 r},$ to see this notice that $\log n \le 10r^{3\ell-1}k;\:\:$ $k^2\left(3\binom{r-1}{\ell-1}\right)^{k}\le k^2r^{\ell k}\le 2^{4k(\ell-1)\log_2r}; \:\:$ $10r^{3\ell-1}\left(1+6\binom{r-1}{\ell-1}\right) \ell^\ell\le r^{5r}\le 2^{3r^2}\le 2^{3\binom{r}{\ell}^2};$ and $2^{16\binom{r}{\ell}^2}\le 2^{4r^{2\ell}}$.
\end{proof}
% it is easy to get
% $n=2^{(c_1\binom{r}{\ell}^2+c_2k\log_2r)\binom{r-1}{\ell-1}},$ for
% absolute constants $c_1,c_2.$

After this paper was submitted Balogh and Samotij obtained a more efficient container lemma in \cite{balogh2}. This can be used to obtain a slight improvement in the bound of the above theorem.

\subsubsection{Lower bounds.}
Let us now turn towards lower bounds.
The main tool is the following lemma, giving us a lower bound for
$\R(H,k)$ in terms of the ordinary Ramsey number, but with fewer
colours. 
%This will allow us to use it to give lower bounds on the list Ramsey
%number, using our knowledge of the ordinary Ramsey numbers.

\begin{thm}\label{thm:lower-bound}
If $R(H,\lfloor k/(\ell \log n)\rfloor)> n$ then:
$$\R(H,k) > n.$$
\end{thm}

\begin{proof}
The proof will proceed along similar lines as that of \Cref{lem:lower-bound-mathcings}. Let $m=\lfloor k/(\ell \log n)\rfloor$. Consider a colouring $c:E\left(K_n^{(\ell)}\right) \to [m],$ without a monochromatic $H,$ which exists because $n < R(H,m)$.

Let each edge $e$ of $K_n^{(\ell)}$ be assigned a list $L_e$ of size $k$, our goal is to show that we can pick colours from the lists avoiding a monochromatic copy of $H$.

We assign to each colour a type from $[m]$, independently and uniformly at random.
Let $B_e$ be the event that no colour in $L_e$ got assigned type $c(e)$. Then 
$$\P(B_e)\le (1-1/m)^k \le (1-\ell\log n/k)^k \le 1/n^\ell. $$

So by the union bound we obtain:
$$\bigcup_{e \in E\left(K_n^{(\ell)}\right)} \P(B_e)\le \binom{n}{\ell} \cdot 1/n^\ell <1.$$
Thus there is an assignment of types for which every $e\in E\left(K_n^{(\ell)}\right)$ has at least one colour of type $c(e)$ in its list and we colour $e$ in one such colour. In this colouring there can be no monochromatic copy of $H$ since otherwise there would be a monochromatic copy of $H$ under $c,$ contradicting our choice of $c.$
\end{proof}

We can now deduce all our lower bounds from the introduction.

\begin{proof}[ of \Cref{thm:hypergraph-lower}] 
Let us first show that $R(H,k) > r^k.$ In order to do this we exhibit 
a colouring of $G=K_{r^k}^{(\ell)}$ without a monochromatic copy of $H.$ We split $G$ into $r$ equal parts and colour all edges not completely within one of the parts using colour $1,$ then we repeat within each of the parts. Notice that since $\chi(H)>r$ there can be no monochromatic copy of $H$ in this colouring, implying the claim.

Choosing $n=r^{\floor{\sqrt{k /(\ell \log r)}}}$ we have that 
$$R(H,\floor{k/(\ell \log n)}) > r^{\floor{k/(\ell\log n)}}\ge r^{\floor{\sqrt{k/(\ell \log r)}}}=n.$$
Hence \Cref{thm:lower-bound} applies, giving us the desired inequality.
\end{proof}

\begin{proof}[ of Theorem~\ref{thm:non-l-partite}]
Axenovich, Gy\'arf\'as, Liu and Mubayi \cite{axenovich} showed 
that if an $\ell$-graph $H$ is not $\ell$-partite then 
\begin{align} %\label{eq:lb2}
R(H,k) \ge e^{k/((\ell+1)e^\ell)}.
\end{align}
Then for $n=\floor{e^{c_\ell\sqrt{k}}}$ we have that 
$$R(G,\floor{k/(\ell \log n)}) \ge e^{\floor{k/(\ell\log n)}/((\ell+1)e^{\ell})}\ge e^{c_\ell\sqrt{k}}\ge n,$$
so \Cref{thm:lower-bound} applies and gives us the desired inequality.
\end{proof}

\begin{proof}[ of Theorem~\ref{thm:l-partite}]
The upper bound is the trivial inequality \eqref{ineq:trivial}. For
the lower bound we set $n=R(H,\floor{ck/\log k})-1,$ which 
implies $\left\lfloor\frac{ck}{\log k} \right\rfloor\ex(H, n)
\geq \binom{n}{\ell}$, since each colour class is $H$-free. 
Using Erd\H os' upper bound~\cite[Theorem 1]{erdos-kovari-sos}  on the Tur\'an
number of $\ell$-partite $\ell$-graphs one obtains
\begin{align}\label{eq:lb1}
R(H,k) \le (k \ell^\ell)^{r^{\ell-1}},
\end{align}
%a matching lower bound for $H=K^{\ell}(r)$ comes from the standard probabilistic argument (giving something like 
%$$\R(H,k) \ge k^{r^{\ell-1}/\ell},$$ 
%but this is not so nice for us since this is (probably) the best known lower bound anyway... 
for any $\ell$-partite $\ell$-graph $H$ with each part of size at most $r$.
%This was observed by Axenovich, Gy\'arf\'as, Liu and Mubayi \cite{axenovich}. 
Substituting 
$1/c:=2r^{\ell-1}\ell^2\log \ell$ we get that $$\lfloor k/(\ell \log n)\rfloor  \ge \lfloor k/(\ell \log (k \ell^\ell)^{r^{\ell-1}})\rfloor \ge \floor{ck/\log k}.$$ So we obtain that $$R(H,\lfloor k/(\ell \log n)\rfloor) \ge R(H,\floor{ck/\log k})>n.$$ Hence, \Cref{thm:lower-bound} implies the result.

To deduce the second part, note that from $\ex(H,n)=\tilde{\Theta}(n^{\ell - \varepsilon(H)})$ it is not hard to deduce that $R(H,k)=\tilde{\Theta}(k^{1/\eps(H)}),$ for example it follows from Lemma 15 of \cite{axenovich}. Combining this and the first part of the theorem the result follows.
\end{proof}

%\remark{It would actually be interesting to maybe obtain a better than exponential bound on normal Ramsey number of non \ell-partite hypergraphs (so for \ell>=3) as then we could claim separation for any non-$r$-partite hypergraphs. (If this is true at all,...) For r=3 this seems to be Problem 28 of \cite{axenovich}.}
\setcounter{thm}{0}
\section{Concluding remarks and open problems}
In this paper we initiate the systematic study of list Ramsey numbers
of graphs and hypergraphs. We obtain several general bounds and 
reach a good understanding of how the
list Ramsey number relates to the ordinary Ramsey number for some
families of graphs.
There are plenty of very natural further questions that arise.

For stars we have shown that the list Ramsey number is at most one
smaller than the Ramsey number. 
We showed that they are equal in the case of two colours or
when the size of the star is sufficiently large compared to the number
of colours. Actually, we could not show them to differ for any values 
of the parameters, and we tend to conjecture that they are
always equal.
\begin{conj}
For any $r,k\in \mathbb N$ 
$$ \R(K_{1,r},k)=R(K_{1,r},k).$$
\end{conj}
Proving this conjecture for small $r$, in particular for
$r=2$, seems to be difficult, since that  is
equivalent to the well-studied and still open List Colouring
Conjecture for cliques. That said, it would also be really interesting to
show the conjecture for any $r \ge 3,$ because this already seems to
require new ideas.

For matchings we determine the list Ramsey number up to a constant
factor. While our approach is very similar to the one we use in the
general setting, we obtain very good bounds by exploiting the very
simple structure of matchings. It would be interesting, but again
probably hard, to determine the list Ramsey number of matchings
exactly. We actually obtain the list Ramsey number of matchings up to
a smaller order additive term when the size of the matching is
sufficiently larger than the number of colours. When the number of
colours is large enough compared to the size then we could obtain
tight bounds only up to a multiplicative constant factor. It would be
highly desirable to prove bounds which are correct up to a lower order term.
\begin{qn}
Does the limit 
$$\lim\limits_{k \to \infty} \R(rK_2,k)/(k/\log k)$$
exist and if it does what is its value?
\end{qn}
If this limit exists we have shown that it is between $r/4$ and $34r$.
While we did not make a serious attempt to optimise these constant
factors and it is not hard to improve them by being more careful 
with our arguments, finding the precise constant factor
seems to require new ideas.

There are many other families of graphs for which pretty good bounds are known for the Ramsey number, such as paths or cycles, and which might exhibit interesting behaviour in the list Ramsey setting.

In the case of general graphs and hypergraphs 
we have shown that the list Ramsey number is bounded above by a single exponential function in terms of the number of colours, which for higher uniformity hypergraphs is in stark contrast to the ordinary Ramsey number, 
which is known to exhibit an iterated exponential behaviour. 
In the case of $\ell$-partite $\ell$-graphs we showed that the list Ramsey
number is in fact a polynomial function of the number of colours and
that it is close to the ordinary Ramsey number. For non $\ell$-partite
$\ell$-graphs we have shown a lower bound which is exponential in the
square root of the number of colours. It would be interesting to
ascertain whether this lower bound or the exponential upper bound is
closer to the truth, even only for some specific families (of
non-$\ell$-partite $\ell$-graphs) such as cliques. In fact for the
case of $\ell=2,$ that is, for graphs, it is still open whether the
$k$-colour list Ramsey number of cliques is always equal to its
ordinary Ramsey counterpart.
\begin{qn}
Is it true that for any $r,k \in \mathbb{N}$ 
$$ \R(K_r,k)=R(K_r,k)?$$
\end{qn}

We have shown how list Ramsey numbers connect to various interesting
problems and sometimes exhibit very different behaviour when compared
to their ordinary Ramsey counterparts. Such information may give some
indication for the original Ramsey problem as well. For example, since $\R(K_3,k) \le (4+o(1))^k$ if one wishes to construct an example showing $R(K_3,k)$ is super-exponential in $k$ (and in the process win a \$100 prize from Erd\H{o}s) one needs to ensure this example does not also work in the case of list Ramsey numbers.

Ramsey theory is very rich in attractive problems and there are many such problems which may prove to be interesting in the list Ramsey setting as well. Some classical examples that come to mind are Schur's or Van der Waerden's Theorems.

\paragraph{Acknowledgements}
The research on this project was initiated during a joint research
workshop of Tel Aviv University and the Free University of Berlin
on
Graph and Hypergraph Colouring Problems, held in Berlin in August
2018,
and supported by a GIF grant number G-1347-304.6/2016. We would
like to
thank the German-Israeli Foundation (GIF) and
both institutions for their support.

We are extremely grateful to the anonymous referees for their many useful suggestions and comments.

\end{document}